\documentclass[12pt,english]{amsart}
\usepackage[T1]{fontenc}
\usepackage[latin9]{inputenc}
\usepackage{color}
\usepackage{babel}
\usepackage{mathrsfs}
\usepackage{amsthm}
\usepackage{amssymb}
\usepackage[unicode=true,
 bookmarks=true,bookmarksnumbered=true,bookmarksopen=true,bookmarksopenlevel=1,
 breaklinks=false,pdfborder={0 0 1},backref=false,colorlinks=true]
 {hyperref}
\hypersetup{pdftitle={The LyX Tutorial},
 pdfauthor={LyX Team},
 pdfsubject={LyX-documentation Tutorial},
 pdfkeywords={LyX, documentation},
 linkcolor=black, citecolor=black, urlcolor=blue, filecolor=blue,pdfpagelayout=OneColumn, pdfnewwindow=true, pdfstartview=XYZ, plainpages=false}

\makeatletter
\numberwithin{equation}{section}
\numberwithin{figure}{section}
\theoremstyle{plain}
\newtheorem{thm}{\protect\theoremname}
  \theoremstyle{definition}
  \newtheorem{defn}[thm]{\protect\definitionname}
  \theoremstyle{remark}
  \newtheorem*{rem*}{\protect\remarkname}
  \theoremstyle{remark}
  \newtheorem{claim}[thm]{\protect\claimname}

\usepackage{ifpdf} 
\ifpdf 

 \IfFileExists{lmodern.sty}{\usepackage{lmodern}}{}

\fi 

\let\myTOC\tableofcontents
\renewcommand\tableofcontents{%
  \frontmatter
  \pdfbookmark[1]{\contentsname}{}
  \myTOC
  \mainmatter }

\def\LyX{\texorpdfstring{%
  L\kern-.1667em\lower.25em\hbox{Y}\kern-.125emX\@}
  {LyX}}

\makeatother

  \providecommand{\claimname}{Claim}
  \providecommand{\definitionname}{Definition}
  \providecommand{\remarkname}{Remark}
\providecommand{\theoremname}{Theorem}

\begin{document}

\title{The periodic complex method in interpolation spaces}

\author{Eliran Avni}

\address{Department of Mathematics, Technion - Israel Institute of Technology,
Haifa 32000, Israel}

\email{eliran5@tx.technion.ac.il}

\keywords{Interpolation spaces, periodic complex method.}

\thanks{The research was supported by a Chester and Taube Hurwitz Foundation
Fellowship.}
\begin{abstract}
In this paper we consider the {}``periodic'' variant of the complex
interpolation method, apparently first studied by Peetre in \cite{peetre}.
Cwikel showed in \cite{cwikel-complex} that using functions with
a given period $i\lambda$ in the complex method construction introduced
and studied by Calder\'on (see \cite{Calderon}), one may construct
the same interpolation spaces as in the {}``regular'' complex method,
up to equivalence of norms. Cwikel also showed in \cite{cwikel-complex}
that the constants of this equivalence will, in some cases, {}``blow
up'' as $\lambda\rightarrow0$. We show that the constants of this
equivalence approach $1$ as $\lambda\rightarrow\infty$. Intuitively,
this means that when applying the complex method of Calder\'on, it
makes a very small difference if one restricts oneself to periodic
functions, provided that the period is very large.
\end{abstract}
\maketitle

\section{Introduction}

In this section we recall the basic definitions of interpolation spaces
and of the complex method of Calder\'on and its periodic variant.
We also prove several auxiliary claims to be used later in this paper.
Most of the definitions in this section appear extensively in the
literature (see, e.g., \cite{BSr-1,Bergh_Lofstrom,bk,kps,T}).
\begin{defn}
Whenever $A_{0},\, A_{1}$ are two Banach spaces that are both continuously
embedded in some topological Hausdorff vector space $\mathcal{A}$
we say that $\left(A_{0},A_{1}\right)$ is a \textbf{Banach couple}.
\end{defn}

\begin{defn}
Whenever $(A_{0},A_{1})$ is a Banach couple, we define $A_{0}+A_{1}$
to be the space of all elements $a\in\mathcal{A}$ such that $a=a_{0}+a_{1}$
for some $a_{j}\in A_{j}$ ($j=0,1)$. This is a Banach space, when
endowed with the following norm: 
\[
\left\Vert a\right\Vert _{A_{0}+A_{1}}=\inf\left\{ \left\Vert a_{0}\right\Vert _{A_{0}}+\left\Vert a_{1}\right\Vert _{A_{1}}\,:\, a_{j}\in A_{j}\,,\, j=0,1\,,\, a=a_{0}+a_{1}\right\} .
\]

\end{defn}

\begin{defn}
For a given Banach couple $\left(A_{0},A_{1}\right)$, the statement
{}``\textbf{$T:(A_{0},A_{1})\rightarrow(A_{0},A_{1})$ is a bounded
linear operator}'' means that $T$ is a linear operator from $A_{0}+A_{1}$
into itself such that the restriction of $T$ to $A_{j}$ is a bounded
operator from $A_{j}$ into itself (for $j=0,1$).
\end{defn}

\begin{rem*}
We remark that if $T:(A_{0},A_{1})\rightarrow(A_{0},A_{1})$ is a
bounded linear operator then automatically $T$ is also a bounded
linear operator from $A_{0}+A_{1}$ into itself, and the following
inequality holds: 
\[
\Vert T\Vert_{A_{0}+A_{1}\to A_{0}+A_{1}}\leq\max\left\{ \Vert T\Vert_{A_{0}\rightarrow A_{0}},\Vert T\Vert_{A_{1}\rightarrow A_{1}}\right\} \,.
\]

\end{rem*}

We recall that $A_{0}\cap A_{1}$ , when endowed with the norm 
\[
\left\Vert x\right\Vert _{A_{0}\cap A{}_{1}}=\mbox{max}\left\{ \left\Vert x\right\Vert _{A_{0}},\left\Vert x\right\Vert _{A_{1}}\right\} ,
\]
 is also a Banach space. This will be relevant in the following definition.
\begin{defn}
A Banach space $A$ satisfying $A_{0}\cap A_{1}\subseteq A\subseteq A_{0}+A_{1}$
where all the inclusions are continuous (that is, when the identity
maps $i:A_{0}\cap A_{1}\rightarrow A$ and $i:A\rightarrow A_{0}+A_{1}$
are continuous) is called an \textbf{intermediate space} of $(A_{0},A_{1})$.
\end{defn}

\begin{defn}
Whenever $\left(A_{0},A_{1}\right)$ is a Banach couple, the statement
{}``\textbf{$A$ is an interpolation space with respect to $(A_{0},A_{1})$}''
is a concise way to say the following: $A$ is an intermediate space
of $\left(A_{0},A_{1}\right)$, and the restriction to $A$ of every
bounded linear operator $T:(A_{0},A_{1})\rightarrow(A_{0},A_{1})$
is a bounded operator from $A$ into itself.
\end{defn}

For future reference, it is convenient to set $\mathbb{S}=\left\{ z\in\mathbb{Z}\,:\,0\leq\mbox{Re}z\leq1\right\} $.
We also denote the interior of $\mathbb{S}$ by $\mathbb{S}^{\circ}$.

Next we define five special spaces, that play an essential role in
the complex method and the periodic complex method (cf. \cite{Calderon,cwikel-complex,peetre}).
We stress that although the definitions we have given in this paper
thus far are applicable for Banach spaces defined over $\mathbb{R}$
or $\mathbb{C}$, from this point onwards we assume all Banach spaces
we deal with are defined over $\mathbb{C}$.
\begin{defn}
The space $\mathscr{F}^{\infty}\left(A_{0},A_{1}\right)$, sometimes
denoted simply by $\mathscr{F}^{\infty}$, consists of all continuous
functions $f:\mathbb{S}\rightarrow A_{0}+A_{1}$ that satisfy these
conditions:\end{defn}
\begin{itemize}
\item The restriction of $f$ to $\left\{ j+iy\,:\, y\in\mathbb{R}\right\} $
is a continuous function into $A_{j}$, for $j=0,1$.
\item The restriction of $f$ to $\mathbb{S}^{\circ}$  is an analytic function
into $A_{0}+A_{1}$. (One of the equivalent definitions of that statement
is that for any bounded linear functional $l:A_{0}+A_{1}\rightarrow\mathbb{C}$,
the composition $l\circ f:\mathbb{S}^{\circ}\rightarrow\mathbb{C}$
is an analytic function.)
\item $\mbox{sup}_{z\in\mathbb{S}}\left\{ \left\Vert f(z)\right\Vert _{A_{0}+A_{1}}\right\} <\infty$.
\item $\mbox{sup}_{j=0,1\,,\, y\in\mathbb{R}}\left\{ \left\Vert f(j+iy)\right\Vert _{A_{j}}\right\} <\infty$.
\end{itemize}
This is a Banach space when endowed with the norm:
\[
\left\Vert f\right\Vert _{\mathscr{F}^{\infty}}=\mbox{sup}_{j=0,1\,,\, y\in\mathbb{R}}\left\{ \left\Vert f(j+iy)\right\Vert _{A_{j}}\right\} \,.
\]

\begin{defn}
The space $\mathscr{F}\left(A_{0},A_{1}\right)$, or simply $\mathscr{F}$,
is the closed subspace of $\mathscr{F}^{\infty}\left(A_{0},A_{1}\right)$
which consists of functions $f\in\mathscr{F}^{\infty}$ that satisfy
$\mbox{lim}_{\left|y\right|\rightarrow\infty}\left\Vert f(j+iy)\right\Vert _{A_{j}}=0$,
for both $j=0$ and $j=1$.
\end{defn}

\begin{defn}
For each fixed $\lambda>0$ the space $\mathscr{F}_{\lambda}\left(A_{0},A_{1}\right)$,
or simply $\mathscr{F}_{\lambda}$, is the closed subspace of $\mathscr{F}^{\infty}\left(A_{0},A_{1}\right)$
which consists of functions $f\in\mathscr{F}^{\infty}$ that are $i\lambda$-periodic
(that is, $f(z)=f(z+i\lambda)$ for all $z\in\mathbb{S}$).
\end{defn}

\begin{defn}
Following Calder\'on in \cite{Calderon}, for each $0<\theta<1$,
we define $\left[A_{0},A_{1}\right]_{\theta}=\left\{ f(\theta)\,:\, f\in\mathscr{F}\left(A_{0},A_{1}\right)\right\} $.
\end{defn}
This is a Banach space, when endowed with the norm:
\begin{equation}
\left\Vert a\right\Vert _{[\theta]}=\mbox{inf}\left\{ \left\Vert f\right\Vert _{\mathscr{F}^{\infty}}\,:\, a=f(\theta)\,,\, f\in\mathscr{F}\left(A_{0},A_{1}\right)\right\} \,.\label{eq:DefNorm}
\end{equation}

\begin{rem*}
It is well known that in this preceding definition of the space $\left[A_{0},A_{1}\right]_{\theta}$
one can replace $\mathscr{F}$ by $\mathscr{F}^{\infty}$ and this
does not change the set $\left[A_{0},A_{1}\right]_{\theta}$. Furthermore,
replacing $\mathscr{F}$ by $\mathscr{F}^{\infty}$ in (\ref{eq:DefNorm})
does not change the norm $\left\Vert a\right\Vert _{[\theta]}$. This
is easily deduced from the fact that for every $f\in\mathscr{F}^{\infty}\left(A_{0},A_{1}\right)$
and every $\delta>0$, the function $f_{\delta}(z)=e^{\delta(z-\theta)^{2}}f(z)$
is an element of $\mathscr{F}\left(A_{0},A_{1}\right)$ and from the
fact that $\mbox{inf}_{\delta>0}\left\{ \left\Vert f_{\delta}\right\Vert _{\mathscr{F}^{\infty}}\right\} \leq\left\Vert f\right\Vert _{\mathscr{F}^{\infty}}$
(cf., e.g., \cite{cwikel-complex}. This is also implicit in \cite{Calderon}.)
\end{rem*}

An essential feature of the complex method, as proven in \cite{Calderon},
is that $\left[A_{0},A_{1}\right]_{\theta}$ is an interpolation space
of $(A_{0},A_{1})$ for any $0<\theta<1$. In other words, given a
Banach couple $\left(A_{0},A_{1}\right)$, Calder\'on devised a mechanism
for constructing a family of interpolation spaces of that couple.
We remark that when setting $A_{0}=L^{p_{0}}(\mu)$ and $A_{1}=L^{p_{1}}(\mu)$
one obtains that $\left[A_{0},A_{1}\right]_{\theta}=L^{p_{\theta}}(\mu)$,
where $\frac{1}{p_{\theta}}=\frac{1-\theta}{p_{0}}+\frac{\theta}{p_{1}}$
(with equality of norms) whenever $p_{0},p_{1}\in[1,\infty]$. This
example of spaces of the form $\left[A_{0},A_{1}\right]_{\theta}$,
which of course comes from Thorin's proof of the Riesz-Thorin Theorem,
was probably the key motivating factor that guided Calder\'on in
his construction.
\begin{rem*}
We note that although there are many other examples of Banach couples
for which concrete descriptions of the spaces $\left[A_{0},A_{1}\right]_{\theta}$
may be found, there are also many other examples, including even {}``natural''
ones, where no concrete descriptions of the spaces $\left[A_{0},A_{1}\right]_{\theta}$
are known.
\end{rem*}

\begin{defn}
As in \cite{cwikel-complex}, for each $0<\theta<1$ and $\lambda>0$,
we define $\left[A_{0},A_{1}\right]_{\theta}^{\lambda}=\left\{ f(\theta)\,:\, f\in\mathscr{F}_{\lambda}\left(A_{0},A_{1}\right)\right\} $.
\end{defn}
This is a Banach space, when endowed with the norm:
\[
\left\Vert a\right\Vert _{[\theta,\lambda]}=\mbox{inf}\left\{ \left\Vert f\right\Vert _{\mathscr{F}^{\infty}}\,:\, a=f(\theta)\,,\, f\in\mathscr{F}_{\lambda}\left(A_{0},A_{1}\right)\right\} \,.
\]

It is easy to see that $\left[A_{0},A_{1}\right]_{\theta}^{\lambda}$
is an interpolation space of $\left(A_{0},A_{1}\right)$ for all $0<\theta<1$
and all $\lambda>0$. Furthermore, the inclusion $\left[A_{0},A_{1}\right]_{\theta}^{\lambda}\subseteq\left[A_{0},A_{1}\right]_{\theta}$
and inequality $\left\Vert a\right\Vert _{[\theta]}\leq\left\Vert a\right\Vert _{[\theta,\lambda]}$
(for all $a\in\left[A_{0},A_{1}\right]_{\theta}^{\lambda}$) are obvious.
In fact, Cwikel proved the following theorem in \cite{cwikel-complex}:
\begin{thm}
\label{thm:Cwikel}For every $0<\theta<1$ and every $\lambda>0$
we have $\left[A_{0},A_{1}\right]_{\theta}=\left[A_{0},A_{1}\right]_{\theta}^{\lambda}$
as two subsets of $A_{0}+A_{1}$. Moreover, there is a $C>0$ such
that $\left\Vert a\right\Vert _{\left[\theta\right]}\leq\left\Vert a\right\Vert _{\left[\theta,\lambda\right]}\leq C\left\Vert a\right\Vert _{\left[\theta\right]}$
for all $a\in\left[A_{0},A_{1}\right]_{\theta}$ (the constant $C$
may depend, in general, on both $\lambda$ and $\theta$).
\end{thm}

\begin{rem*}
Cwikel also showed in \cite{cwikel-complex} that no uniform upper
bound of $C$ (with respect to $\lambda$) may be found, since for
some Banach couples $\left(A_{0},A_{1}\right)$, some $\theta\in(0,1)$
and some $a\in\left[A_{0},A_{1}\right]_{\theta}$ we have 
\[
\mbox{limsup}_{\lambda\rightarrow0}\frac{\left\Vert a\right\Vert _{\left[\theta,\lambda\right]}}{\left\Vert a\right\Vert _{\left[\theta\right]}}=\infty\,.
\]
Despite this dependence on $\lambda$ it can be seen (cf. also Theorem
\ref{thm:Main Result} below) that for each fixed $\lambda$ the constant
$C$ can in fact be bounded independently of $\theta$.
\end{rem*}

\begin{rem*}
As mentioned in \cite{cwikel-complex}, the periodic complex interpolation
spaces can also be defined via consideration of analytic Banach space
valued functions defined on an annulus. Let us explain this in a little
more detail: For each $\lambda\in(0,\infty)$ we set $\mathbb{A}_{\lambda}=\left\{ z\in\mathbb{C}\,:\:1\leq\left|z\right|\leq e^{\frac{2\pi}{\lambda}}\right\} $.
We notice that for each $g:\mathbb{A}_{\lambda}\rightarrow A_{0}+A_{1}$
we can define $f_{g}:\mathbb{S}\rightarrow A_{0}+A_{1}$ by setting
$f_{g}(z)=g(e^{\frac{2\pi}{\lambda}z})$. This defines a bijection
map between the set of all $A_{0}+A_{1}$-valued functions on $\mathbb{A}_{\lambda}$
and the set of all $A_{0}+A_{1}$-valued functions on $\mathbb{S}$
which are $i\lambda$-periodic. This map enables us to formulate an
alternative equivalent definition of $\left[A_{0},A_{1}\right]_{\theta}^{\lambda}$
by using $A_{0}+A_{1}$-valued functions defined on $\mathbb{A}_{\lambda}$
in an obvious way.
\end{rem*}

We conclude this section by proving three claims which will be used
in the proof of our main result. 
\begin{claim}
\label{clm: blocking the sum-1}Let $\lambda$ be a fixed positive
number. Setting $C_{1,\lambda}(\alpha)=e^{\alpha}+2e^{\alpha-0.25\alpha\lambda^{2}}\cdot\frac{2-e^{-\alpha\lambda^{2}}}{1-e^{-\alpha\lambda^{2}}}$
for any $\alpha>0$, one has
\[
\mbox{sup}_{\left|y\right|\leq\frac{\lambda}{2}}\left\{ \sum_{k\in\mathbb{Z}}e^{\alpha\left(1-(y+k\lambda)^{2}\right)}\right\} \leq C_{1,\lambda}(\alpha)\,.
\]
\end{claim}
\begin{proof}
First we shall show that 
\begin{equation}
\sum_{k=1}^{\infty}e^{\alpha\left(1-(y+k\lambda)^{2}\right)}\leq e^{\alpha-0.25\alpha\lambda^{2}}\cdot\frac{2-e^{-\alpha\lambda^{2}}}{1-e^{-\alpha\lambda^{2}}}\,.\label{eq:chopin}
\end{equation}
We notice that for every $k\geq1$ and $\left|y\right|\leq\frac{\lambda}{2}$
we have $y+k\lambda\geq k\lambda-\frac{\lambda}{2}\geq\frac{\lambda}{2}>0$,
hence

\begin{eqnarray*}
\sum_{k=1}^{\infty}e^{\alpha\left(1-(y+k\lambda)^{2}\right)} & \leq & \sum_{k=1}^{\infty}e^{\alpha\left(1-(k\lambda-\frac{\lambda}{2})^{2}\right)}\\
 & = & \sum_{k=1}^{\infty}e^{\alpha\left(1-\lambda^{2}(k-\frac{1}{2})^{2}\right)}\\
 & = & e^{\alpha}\left(e^{-\alpha\left(\frac{\lambda}{2}\right)^{2}}+\sum_{k=2}^{\infty}e^{-\alpha\lambda^{2}(k-\frac{1}{2})^{2}}\right)\\
 & \leq & e^{\alpha}\left(e^{-\alpha\left(\frac{\lambda}{2}\right)^{2}}+\sum_{k=2}^{\infty}e^{-\alpha\lambda^{2}(k-\frac{1}{2})}\right)\\
 & = & e^{\alpha}\left(e^{-\alpha\left(\frac{\lambda}{2}\right)^{2}}+e^{0.5\alpha\lambda^{2}}\sum_{k=2}^{\infty}e^{-\alpha\lambda^{2}k}\right)\\
 & = & e^{\alpha}\left(e^{-\alpha\left(\frac{\lambda}{2}\right)^{2}}+e^{0.5\alpha\lambda^{2}}\cdot\frac{e^{-2\alpha\lambda^{2}}}{1-e^{-\alpha\lambda^{2}}}\right)\\
 & \leq & e^{\alpha}\left(e^{-0.25\alpha\lambda^{2}}+\frac{e^{-0.25\alpha\lambda^{2}}}{1-e^{-\alpha\lambda^{2}}}\right)\\
 & = & e^{\alpha-0.25\alpha\lambda^{2}}\cdot\frac{2-e^{-\alpha\lambda^{2}}}{1-e^{-\alpha\lambda^{2}}}\,.
\end{eqnarray*}
This establishes (\ref{eq:chopin}). Then, replacing $y$ by $-y$
and $k$ by $-k$ in (\ref{eq:chopin}), we can immediately deduce
that for every $\left|y\right|\leq\frac{\lambda}{2}$ we also have
\begin{equation}
\sum_{k=-\infty}^{-1}e^{\alpha\left(1-(y+k\lambda)^{2}\right)}\leq e^{\alpha-0.25\alpha\lambda^{2}}\cdot\frac{2-e^{-\alpha\lambda^{2}}}{1-e^{-\alpha\lambda^{2}}}\,.\label{eq: est. neg. sigma-1-1}
\end{equation}
Finally, when $k=0$ and $\left|y\right|\leq\frac{\lambda}{2}$ we
may write 
\begin{equation}
e^{\alpha\left(1-(y+k\lambda)^{2}\right)}=e^{\alpha\left(1-y^{2}\right)}\leq e^{\alpha}\,.\label{eq: est. k=00003D0-1-1}
\end{equation}
Combining (\ref{eq:chopin}), (\ref{eq: est. neg. sigma-1-1}), and
(\ref{eq: est. k=00003D0-1-1}) we obtain
\[
\begin{array}{rcl}
{\displaystyle \mbox{sup}_{\left|y\right|\leq\frac{\lambda}{2}}\left\{ \sum_{k\in\mathbb{Z}}e^{\alpha\left(1-(y+k\lambda)^{2}\right)}\right\} } & \le & {\displaystyle e^{\alpha}+2e^{\alpha-0.25\alpha\lambda^{2}}\cdot\frac{2-e^{-\alpha\lambda^{2}}}{1-e^{-\alpha\lambda^{2}}}}\\
 & = & C_{1,\lambda}(\alpha)
\end{array}
\]
completing the proof.\end{proof}
\begin{claim}
\label{Clm:Mozart}Let $\lambda$ be a fixed positive number. For
every $f\in\mathscr{F}\left(A_{0},A_{1}\right)$, every $\alpha>0$,
every $0<\theta<1$ and every analytic function $g:\mathbb{C}\rightarrow\mathbb{C}$,
if $\mbox{sup}_{z\in\mathbb{S}}\left\{ \left|g(z)\right|\right\} =M<\infty$
then the function $F_{\alpha}:\mathbb{S}\rightarrow A_{0}+A_{1}$
defined by
\[
F_{\alpha}(z)=\sum_{k\in\mathbb{Z}}e^{\alpha(z-\theta+ik\lambda)^{2}}g(z+ik\lambda)\cdot f(z+ik\lambda)
\]
is an element of $\mathscr{F}_{\lambda}\left(A_{0},A_{1}\right)$.
Moreover, we also have
\[
\left\Vert F_{\alpha}\right\Vert _{\mathscr{F}^{\infty}}\leq C_{1,\lambda}(\alpha)M\cdot\left\Vert f\right\Vert _{\mathscr{F}^{\infty}}\,.
\]
\end{claim}
\begin{proof}
It is clear that the scalar valued series $\sum_{k\in\mathbb{Z}}e^{\alpha(z-\theta+ik\lambda)^{2}}g(z+ik\lambda)$
converges absolutely and uniformly in every compact subset of $\mathbb{S}$.
This fact implies, first of all, that the $A_{0}+A_{1}$-valued function
$F_{\alpha}$ is well defined, continuous, and clearly $i\lambda$-periodic
and therefore bounded. The same fact also implies that the restriction
of $F_{\alpha}$ to $\mathbb{S}^{\circ}$ is analytic. Defining $\Lambda_{j}=\left\{ j+iy\,:\, y\in\mathbb{R}\right\} $
for $j=0,1$, the fact that the functions $F_{\alpha}\mid_{\Lambda_{j}}:\Lambda_{j}\rightarrow A_{j}$
(for $j=0,1$) are continuous (and bounded due to periodicity) follows
readily from the same argument.

Thus we have established that $F_{\alpha}\in\mathscr{F}_{\lambda}\left(A_{0},A_{1}\right)$.
Furthermore, bearing in mind that $F_{\alpha}$ is $i\lambda$-periodic,
we see that
\begin{equation}
\left\Vert F_{\alpha}\right\Vert _{\mathscr{F}^{\infty}}=\mbox{sup}_{j=0,1\,,\,\left|y\right|\leq\frac{\lambda}{2}}\left\{ \left\Vert F_{\alpha}(j+iy)\right\Vert _{A_{j}}\right\} \,.\label{eq:F_alpha_norm-1}
\end{equation}
Now, for each fixed $j\in\left\{ 0,1\right\} $ and $\left|y\right|\leq\frac{\lambda}{2}$
we have
\begin{equation}
\begin{array}{cl}
 & {\displaystyle \left\Vert F_{\alpha}(j+iy)\right\Vert _{A_{j}}}\\
= & {\displaystyle \mbox{lim}_{n\rightarrow\infty}\left\Vert \sum_{k=-n}^{n}e^{\alpha(j+i(y+k\lambda)-\theta)^{2}}g(j+i(y+k\lambda))\cdot f(j+i(y+k\lambda))\right\Vert _{A_{j}}}\\
\le & {\displaystyle \mbox{lim}_{n\rightarrow\infty}\sum_{k=-n}^{n}\left|e^{\alpha(j+i(y+k\lambda)-\theta)^{2}}g(j+i(y+k\lambda))\right|\left\Vert f(j+i(y+k\lambda))\right\Vert _{A_{j}}}\\
\le & {\displaystyle \mbox{lim}_{n\rightarrow\infty}\sum_{k=-n}^{n}e^{\alpha\left((j-\theta)^{2}-(y+k\lambda)^{2}\right)}M\left\Vert f\right\Vert _{\mathscr{F}^{\infty}}}\\
\le & {\displaystyle {\displaystyle \mbox{lim}_{n\rightarrow\infty}\sum_{k=-n}^{n}e^{\alpha\left(1-(y+k\lambda)^{2}\right)}M\left\Vert f\right\Vert _{\mathscr{F}^{\infty}}\,.}}
\end{array}\label{eq:NEW FalphaOnTheBoundary}
\end{equation}
Combining (\ref{eq:F_alpha_norm-1}), (\ref{eq:NEW FalphaOnTheBoundary})
and Claim \ref{clm: blocking the sum-1} we conclude that $\left\Vert F_{\alpha}\right\Vert _{\mathscr{F}^{\infty}}\leq C_{1,\lambda}(\alpha)M\cdot\left\Vert f\right\Vert _{\mathscr{F}^{\infty}}$,
as required.
\end{proof}

\begin{claim}
\label{clm:bounding_w-1}For a given $\lambda>0$ and $0<\theta<1$,
let us define $w:\mathbb{C}\to\mathbb{C}$ by $w(\theta)=1$ and $w(z)=\frac{\lambda}{2\pi}\cdot\frac{e^{\frac{2\pi}{\lambda}(z-\theta)-1}}{z-\theta}$
for all $z\ne\theta$. We also define $m(\lambda)=\frac{\lambda}{2\pi}\left(1+e^{\frac{4\pi}{\lambda}}\right)$.
We then have
\begin{equation}
\mbox{sup}_{z\in\mathbb{S}}\left\{ \left|w(z)\right|\right\} \leq m(\lambda)\,.\label{eq:upper_bd_for_w-1}
\end{equation}
\end{claim}
\begin{proof}
It is easy to see that $w(z)$ is an analytic function in the entire
complex plane, bounded on $\mathbb{S}_{2}:=\left\{ z\in\mathbb{C}\,:\:-1\leq\mbox{Re}z\leq2\right\} $.
Applying the Phragm\'en-Lindel\"of maximum principle we deduce that
\begin{eqnarray*}
\mbox{sup}_{z\in\mathbb{S}}\left\{ \left|w(z)\right|\right\}  & \leq & \mbox{sup}_{z\in\mathbb{S}_{2}}\left\{ \left|w(z)\right|\right\} \\
 & = & \mbox{sup}_{z\in\partial\mathbb{S}_{2}}\left\{ \left|w(z)\right|\right\} \\
 & = & \mbox{sup}_{j=-1,2\,,\, y\in\mathbb{R}}\left\{ \left|w(j+iy)\right|\right\} \,.
\end{eqnarray*}
According to the definition of $w(z)$, for $j=-1,2$ and for all
$y\in\mathbb{R}$, one has
\[
\begin{array}{rcl}
{\displaystyle \left|w(j+iy)\right|} & = & {\displaystyle {\displaystyle \left|\frac{\lambda}{2\pi}\cdot\frac{\left(e^{\frac{2\pi}{\lambda}(j-\theta+iy)}-1\right)}{j-\theta+iy}\right|}}\\
 & \leq & {\displaystyle \frac{\lambda}{2\pi\left|j-\theta\right|}\left(\left|e^{\frac{2\pi}{\lambda}(j-\theta+iy)}\right|+1\right)}\\
 & = & {\displaystyle \frac{\lambda}{2\pi\left|j-\theta\right|}\left(1+e^{\frac{2\pi}{\lambda}(j-\theta)}\right)}\\
 & \leq & {\displaystyle {\displaystyle \frac{\lambda}{2\pi\left|j-\theta\right|}\left(1+e^{\frac{4\pi}{\lambda}}\right)\,.}}
\end{array}
\]
Since $j=-1,2$ and $0<\theta<1$ we necessarily have $\frac{1}{\left|j-\theta\right|}\leq1$,
and from this (\ref{eq:upper_bd_for_w-1}) immediately follows.
\end{proof}

\section{The Main Result}

In this section we state and prove our main result which complements
the results obtained in \cite{cwikel-complex}.
\begin{thm}
\label{thm:Main Result}There is a function $C=C(\lambda)$ such that
$\mbox{lim}_{\lambda\rightarrow\infty}C(\lambda)=1$ and, for all
Banach couples $(A_{0},A_{1})$, all $0<\theta<1$, all $a\in\left[A_{0},A_{1}\right]_{\theta}$,
and all $\lambda>0$, one has
\begin{equation}
\left\Vert a\right\Vert _{[\theta]}\le\left\Vert a\right\Vert _{[\theta,\lambda]}\le C(\lambda)\left\Vert a\right\Vert _{[\theta]}\,.\label{eq:main_inequality}
\end{equation}

\end{thm}
As we have mentioned, the equality $\left[A_{0},A_{1}\right]_{\theta}=\left[A_{0},A_{1}\right]_{\theta}^{\lambda}$
was already proven in Theorem \ref{thm:Cwikel}, and the inequality
$\left\Vert a\right\Vert _{[\theta]}\le\left\Vert a\right\Vert _{[\theta,\lambda]}$
obviously holds for all $a\in\left[A_{0},A_{1}\right]_{\theta}=\left[A_{0},A_{1}\right]_{\theta}^{\lambda}$,
so in order to prove Theorem \ref{thm:Main Result} we need only to
prove the right hand side of (\ref{eq:main_inequality}). The proof
extends over the next three subsections:

\subsection{A rough estimate}

Suppose $x\in\left[A_{0},A_{1}\right]_{\theta}$, and $f\in\mathscr{F}(A_{0},A_{1})$
is such that $x=f(\theta)$. Let us fix $\lambda>0$ and, following
Cwikel in \cite{cwikel-complex}, define $G_{\delta}(z)=\sum_{k\in\mathbb{Z}}e^{\delta(z-\theta+ik\lambda)^{2}}w(z+ik\lambda)\cdot f(z+ik\lambda)$
where we leave the value of $\delta>0$ undetermined for now (the
function $w(z)$ was defined in Claim \ref{clm:bounding_w-1}). Applying
Claim \ref{clm:bounding_w-1} and then Claim \ref{Clm:Mozart}, we
see that $G_{\delta}\in\mathscr{F}_{\lambda}\left(A_{0},A_{1}\right)$
for every $\delta>0$. Moreover, since $w(\theta)=1$ and $w(\theta+ik\lambda)=0$
for all $0\neq k\in\mathbb{Z}$, we have $G_{\delta}(\theta)=f(\theta)=x$.
Again, applying Claims \ref{clm:bounding_w-1} and \ref{Clm:Mozart},
we gather that 
\begin{eqnarray*}
\left\Vert x\right\Vert _{[\theta,\lambda]} & \leq & \left\Vert G_{\delta}\right\Vert _{\mathscr{F}^{\infty}}\\
 & \leq & m(\lambda)C_{1,\lambda}(\delta)\left\Vert f\right\Vert _{\mathscr{F}^{\infty}}\,.
\end{eqnarray*}
Finally, taking the infimum over all $f\in\mathscr{F}\left(A_{0},A_{1}\right)$
such that $x=f(\theta)$ will yield

\begin{eqnarray}
\left\Vert x\right\Vert _{[\theta,\lambda]} & \leq & m(\lambda)C_{1,\lambda}(\delta)\left\Vert x\right\Vert _{[\theta]}\label{eq:crude estimation of equi. constant}
\end{eqnarray}
for all $x\in\left[A_{0},A_{1}\right]_{\theta}$.

\subsection{Some fine tuning}

Given an $a\in\left[A_{0},A_{1}\right]_{\theta}$, for each $f\in\mathscr{F}\left(A_{0},A_{1}\right)$
such that $a=f(\theta)$ we choose a certain $\rho>0$ to be determined
later and define $\tilde{f}_{\rho}(z)=\sum_{k\in\mathbb{Z}}e^{\rho(z-\theta+ik\lambda)^{2}}f(z+ik\lambda)$.
According to Claim \ref{Clm:Mozart} (when we set $g(z)\equiv1$)
we have $\tilde{f}_{\rho}\in\mathscr{F}_{\lambda}\left(A_{0},A_{1}\right)$.
Defining $\tilde{a}=\tilde{f}_{\rho}(\theta)\in\left[A_{0},A_{1}\right]_{\theta}^{\lambda}$,
we have, according to the same claim, that
\begin{equation}
\begin{array}{rcl}
{\displaystyle \left\Vert \tilde{a}\right\Vert _{[\theta,\lambda]}} & \leq & {\displaystyle \left\Vert \tilde{f}_{\rho}\right\Vert _{\mathscr{F}^{\infty}}}\\
 & \leq & {\displaystyle C_{1,\lambda}(\rho)\left\Vert f\right\Vert _{\mathscr{F}^{\infty}}\,.}
\end{array}\label{eq: 1st est. of a tilde}
\end{equation}
Defining 
\[
P_{n}=\left\Vert \tilde{a}-{\displaystyle \sum_{k=-n}^{n}}e^{-\rho(k\lambda)^{2}}f(\theta+ik\lambda)\right\Vert _{[\theta]}
\]
and
\[
Q_{n}=\left\Vert {\displaystyle \sum_{k=-n}^{n}}e^{-\rho(k\lambda)^{2}}f(\theta+ik\lambda)-a\right\Vert _{[\theta]}
\]
we obviously may write
\begin{equation}
\left\Vert \tilde{a}-a\right\Vert _{[\theta]}\le P_{n}+Q_{n}\,.\label{eq:a tilde a Qn Pn}
\end{equation}
Estimating the expression $Q_{n}$ we get

\begin{eqnarray}
Q_{n} & = & {\displaystyle \left\Vert \sum_{k=-n}^{n}e^{-\rho(k\lambda)^{2}}f(\theta+ik\lambda)-a\right\Vert _{[\theta]}}\nonumber \\
 & = & \left\Vert \sum_{k=-n}^{n}e^{-\rho(k\lambda)^{2}}f(\theta+ik\lambda)-f(\theta)\right\Vert _{[\theta]}\nonumber \\
 & \le & \sum_{k=-n}^{-1}\left\Vert e^{-\rho(k\lambda)^{2}}f(\theta+ik\lambda)\right\Vert _{[\theta]}+\sum_{k=1}^{n}\left\Vert e^{-\rho(k\lambda)^{2}}f(\theta+ik\lambda)\right\Vert _{[\theta]}\nonumber \\
 & \le & \sum_{k=-n}^{-1}e^{-\rho(k\lambda)^{2}}\left\Vert f\right\Vert _{\mathscr{F}^{\infty}}+\sum_{k=1}^{n}e^{-\rho(k\lambda)^{2}}\left\Vert f\right\Vert _{\mathscr{F}^{\infty}}\label{eq:NEW-Qn}\\
 & = & {\displaystyle 2\sum_{k=1}^{n}e^{-\rho(k\lambda)^{2}}\left\Vert f\right\Vert _{\mathscr{F}^{\infty}}}\nonumber \\
 & \le & 2\sum_{k=1}^{\infty}e^{-\rho k\lambda^{2}}\left\Vert f\right\Vert _{\mathscr{F}^{\infty}}\nonumber \\
 & = & {\displaystyle 2e^{-\rho\lambda^{2}}\cdot\frac{1}{1-e^{-\rho\lambda^{2}}}\left\Vert f\right\Vert _{\mathscr{F}^{\infty}}}\,.\nonumber 
\end{eqnarray}

We will now show that 
\begin{equation}
\mbox{lim}_{n\rightarrow\infty}P_{n}=0\,.\label{eq:Pn}
\end{equation}
First we note that, by definition 
\begin{equation}
\mbox{lim}_{n\rightarrow\infty}\left\Vert \tilde{a}-\sum_{k=-n}^{n}e^{-\rho(k\lambda)^{2}}f(\theta+ik\lambda)\right\Vert _{A_{0}+A_{1}}=0\,.\label{eq: the eq. we know-1}
\end{equation}
Furthermore, we observe that $\left(\sum_{k=-n}^{n}e^{-\rho(k\lambda)^{2}}f(\theta+ik\lambda)\right)_{n\in\mathbb{N}}$
is a Cauchy sequence in $\left[A_{0},A_{1}\right]_{\theta}.$ Indeed,
if $m>n$ then 
\begin{eqnarray*}
{\color{red}} & {\color{red}} & \left\Vert \sum_{k=-m}^{m}e^{-\rho(k\lambda)^{2}}f(\theta+ik\lambda)-\sum_{k=-n}^{n}e^{-\rho(k\lambda)^{2}}f(\theta+ik\lambda)\right\Vert _{[\theta]}\\
{\color{red}} & \le & \sum_{k=n+1}^{m}\left\Vert e^{-\rho(k\lambda)^{2}}f(\theta+ik\lambda)\right\Vert _{[\theta]}+\sum_{k=-m}^{-n-1}\left\Vert e^{-\rho(k\lambda)^{2}}f(\theta+ik\lambda)\right\Vert _{[\theta]}\\
{\color{red}} & \leq & 2\sum_{k=n+1}^{\infty}e^{-\rho\lambda^{2}k^{2}}\left\Vert f\right\Vert _{\mathscr{F}^{\infty}}
\end{eqnarray*}
and since clearly $\mbox{lim}_{n\rightarrow\infty}2\sum_{k=n+1}^{\infty}e^{-\rho\lambda^{2}k^{2}}\left\Vert f\right\Vert _{\mathscr{F}^{\infty}}=0$,
the sequence $\left(\sum_{k=-n}^{n}e^{-\rho(k\lambda)^{2}}f(\theta+ik\lambda)\right)_{n\in\mathbb{N}}$
is Cauchy. This means that there is an element $\bar{a}\in\left[A_{0},A_{1}\right]_{\theta}$
such that 
\begin{equation}
\mbox{lim}_{n\rightarrow\infty}\left\Vert \bar{a}-\sum_{k=-n}^{n}e^{-\rho(k\lambda)^{2}}f(\theta+ik\lambda)\right\Vert _{[\theta]}=0\,.\label{eq:almost there}
\end{equation}
Since $\left\Vert x\right\Vert _{A_{0}+A_{1}}\leq\left\Vert x\right\Vert _{[\theta]}$
for every $x\in\left[A_{0},A_{1}\right]_{\theta}$ (cf., e.g., \cite{Calderon}
paragraph 23 p.~129 where there is implicitly a routine application
of the Phragm\'en-Lindel\"of maximum principle for vector valued
analytic functions), it is also true that
\[
\mbox{lim}_{n\rightarrow\infty}\left\Vert \bar{a}-\sum_{k=-n}^{n}e^{-\rho(k\lambda)^{2}}f(\theta+ik\lambda)\right\Vert _{A_{0}+A_{1}}=0\,.
\]
Combining this last equality with (\ref{eq: the eq. we know-1}) we
immediately get that $\bar{a}=\tilde{a}$. Using (\ref{eq:almost there}),
we immediately obtain (\ref{eq:Pn}).

Combining (\ref{eq:NEW-Qn}), (\ref{eq:Pn}) and (\ref{eq:a tilde a Qn Pn})
we get
\begin{equation}
\left\Vert \tilde{a}-a\right\Vert _{[\theta]}\le\frac{2e^{-\rho\lambda^{2}}}{1-e^{-\rho\lambda^{2}}}\left\Vert f\right\Vert _{\mathscr{F}^{\infty}}\,.\label{eq:differ a and a tilde}
\end{equation}

\subsection{Conclusion}

Combining (\ref{eq:crude estimation of equi. constant}), (\ref{eq: 1st est. of a tilde}),
and (\ref{eq:differ a and a tilde}) we see that for each $a\in\left[A_{0},A_{1}\right]_{\theta}^{\lambda}$
and each $f\in\mathscr{F}\left(A_{0},A_{1}\right)$ such that $a=f(\theta)$
we may write
\begin{eqnarray*}
\left\Vert a\right\Vert _{\left[\theta,\lambda\right]} & \leq & \left\Vert a-\tilde{a}\right\Vert _{\left[\theta,\lambda\right]}+\left\Vert \tilde{a}\right\Vert _{\left[\theta,\lambda\right]}\\
 & \leq & m(\lambda)C_{1,\lambda}(\delta)\left\Vert a-\tilde{a}\right\Vert _{\left[\theta\right]}+\left\Vert \tilde{a}\right\Vert _{\left[\theta,\lambda\right]}\\
 & \leq & m(\lambda)C_{1,\lambda}(\delta)\cdot\frac{2e^{-\rho\lambda^{2}}}{1-e^{-\rho\lambda^{2}}}\left\Vert f\right\Vert _{\mathscr{F}^{\infty}}+C_{1,\lambda}(\rho)\left\Vert f\right\Vert _{\mathscr{F}^{\infty}}\,.
\end{eqnarray*}
Taking the infimum over all $f\in\mathscr{F}\left(A_{0},A_{1}\right)$
for which $a=f(\theta)$ we conclude that
\[
\left\Vert a\right\Vert _{\left[\theta,\lambda\right]}\leq\left(m(\lambda)\cdot\frac{2e^{-\rho\lambda^{2}}}{1-e^{-\rho\lambda^{2}}}C_{1,\lambda}(\delta)+C_{1,\lambda}(\rho)\right)\left\Vert a\right\Vert _{\left[\theta\right]}\,.
\]
If we now choose $\rho=\delta=\frac{1}{\lambda}$ we see that 
\[
\left\Vert a\right\Vert _{\left[\theta,\lambda\right]}\leq\left(m(\lambda)\cdot\frac{2e^{-\lambda}}{1-e^{-\lambda}}+1\right)C_{1,\lambda}\left(\frac{1}{\lambda}\right)\left\Vert a\right\Vert _{\left[\theta\right]}
\]
for all $a\in\left[A_{0},A_{1}\right]_{\theta}$. Since
\[
m(\lambda)\cdot\frac{2e^{-\lambda}}{1-e^{-\lambda}}=\frac{\lambda}{2\pi}\left(1+e^{\frac{4\pi}{\lambda}}\right)\frac{2e^{-\lambda}}{1-e^{-\lambda}}
\]
and
\[
C_{1,\lambda}\left(\frac{1}{\lambda}\right)=e^{\frac{1}{\lambda}}+2e^{\frac{1}{\lambda}-0.25\lambda}\cdot\frac{2-e^{-\lambda}}{1-e^{-\lambda}}
\]
it is easy to see that
\[
\mbox{lim}_{\lambda\rightarrow\infty}m(\lambda)\cdot\frac{2e^{-\lambda}}{1-e^{-\lambda}}=0
\]
and
\[
\mbox{lim}_{\lambda\rightarrow\infty}C_{1,\lambda}\left(\frac{1}{\lambda}\right)=1
\]
hence
\[
\mbox{lim}_{\lambda\rightarrow\infty}\left(m(\lambda)\cdot\frac{2e^{-\lambda}}{1-e^{-\lambda}}+1\right)C_{1,\lambda}\left(\frac{1}{\lambda}\right)=1\,.
\]
Obviously, this completes the proof of Theorem \ref{thm:Main Result}.
\begin{rem*}
We remark that some of our estimates in this paper are quite crude.
We therefore do not claim that the expression 
\[
\left(m(\lambda)\cdot\frac{2e^{-\lambda}}{1-e^{-\lambda}}+1\right)C_{1,\lambda}\left(\frac{1}{\lambda}\right)
\]
which we obtained in the previous proof provides a tight upper bound
of the asymptotic behavior of the best constant $C=C(\lambda)$ which
may be written in (\ref{eq:main_inequality}) as $\lambda\rightarrow\infty$.
\end{rem*}
\textit{Acknowledgements: }We thank Michael Cwikel for helpful discussions
and for his support in the preparation of this paper.

\end{document}